\documentclass[12pt,twoside]{amsart}
\usepackage{amsmath}
\usepackage{amsthm}
\usepackage{amsfonts}
\usepackage{amssymb}
\usepackage{latexsym}
\usepackage{mathrsfs}
\usepackage{amsmath}
\usepackage{amsthm}
\usepackage{amsfonts}
\usepackage{amssymb}
\usepackage{latexsym}
\usepackage{geometry}
\usepackage{dsfont}
\usepackage[dvips]{graphicx}
\usepackage{color}
\usepackage[all]{xy}

\date{}
\allowdisplaybreaks[4] \footskip=15pt
\renewcommand{\uppercasenonmath}[1]{}

\usepackage{graphicx,amssymb}
\usepackage[all]{xy}
\usepackage{amsmath}

\allowdisplaybreaks
\usepackage{amsthm}
\usepackage{color}

\theoremstyle{plain}
\newtheorem{theorem}{Theorem}[section]
\newtheorem{proposition}[theorem]{Proposition}
\newtheorem{lemma}[theorem]{Lemma}
\newtheorem{corollary}[theorem]{Corollary}
\theoremstyle{definition}
\newtheorem{example}[theorem]{Example}
\newtheorem{definition}[theorem]{Definition}

\theoremstyle{definition}
\newtheorem*{acknowledgement}{Acknowledgement}

\theoremstyle{remark}
\newtheorem{remark}[theorem]{Remark}

\newcommand{\pf}{\noindent\begin {proof}}
\newcommand{\epf}{\end{proof}}

\newcommand{\Ker}{\mbox{\rm Ker}}
\newcommand{\Ext}{\mbox{\rm Ext}}
\newcommand{\Hom}{\mbox{\rm Hom}}
\newcommand{\Tor}{\mbox{\rm Tor}}

\newcommand{\Prufer}{Pr\"{u}fer}

\def\m{{\frak m}}

\def\p{{\frak p}}

\def\GV{{\rm GV}}
\def\tor{{\rm tor_{\rm GV}}}
\def\Hom{{\rm Hom}}
\def\Ext{{\rm Ext}}
\def\Tor{{\rm Tor}}

\def\fkm{{\frak m}}

\def\Ker{{\rm Ker}}

\def\Im{{\rm Im}}
\def\Coker{{\rm Coker}}

\def\Nil{{\rm Nil}}

\def\NP{{\rm NP}}
\def\Z{{\rm Z}}

\def\T{{\rm T}}
\def\GV{{\rm GV}}
\def\Max{{\rm Max}}

\def\ZN{{\rm ZN}}

\def\PvMR{{\rm PvMR}}
\def\PvMD{{\rm PvMD}}

\def\Prufer{{\rm Pr\"{u}fer}}
\def\Lt{{\rm L}}
\begin{document}
\begin{center}
{\large  \bf A note on  $\phi$-\Prufer\ $v$-multiplication rings}

\vspace{0.5cm}   Xiaolei Zhang$^{a}$%\ \ \ \ Wei Qi$^b$\\
%\bigskip

{\footnotesize a.\ Department of Basic Courses, Chengdu Aeronautic Polytechnic, Chengdu 610100, China\\
%b.\ School of Mathematical Sciences, Sichuan Normal University, Chengdu 610068, China\\

E-mail: zxlrghj@163.com\\}
\end{center}

\bigskip
\centerline { \bf  Abstract} In this note, we show that  a strongly $\phi$-ring $R$ is a $\phi$-$\PvMR$ if and only if  any $\phi$-torsion free $R$-module is $\phi$-$w$-flat, if and only if    any  divisible module is  nonnil-absolutely $w$-pure module, if and only if     any $h$-divisible module is  nonnil-absolutely $w$-pure module, if and only if  any finitely generated nonnil ideal of $R$ is $w$-projective.
\vbox to 0.3cm{}

\leftskip0truemm \rightskip0truemm
\bigskip

\section{Introduction}

Throughout this note, $R$ denotes a commutative ring with identity and all modules are unitary. We always denote by $\Nil(R)$  the nilpotent radical of $R$, $\Z(R)$ the set of all zero-divisors of $R$ and $\T(R)$ the total ring of fractions of  $R$. An ideal $I$ of $R$  is said to be nonnil if there is a non-nilpotent element in $I$.
A ring $R$ is  an \emph{$\NP$-ring} if $\Nil(R)$ is a prime ideal, and a \emph{$\ZN$-ring} if $\Z(R)=\Nil(R)$. A prime ideal $\p$ is said to be \emph{divided prime} if $\p\subsetneq (x)$, for every $x\in R-\p$. Set $\mathcal{H}=\{R|R$ is a commutative ring and \Nil($R$)\ is a divided prime ideal of $R\}$. A ring $R$ is a \emph{$\phi$-ring} if $R\in \mathcal{H}$. Moreover, a $\ZN$ $\phi$-ring is said to be a \emph{strongly $\phi$-ring}. For a $\phi$-ring $R$, the map $\phi: \T(R)\rightarrow R_{\Nil(R)}$ such that $\phi(\frac{b}{a})=\frac{b}{a}$ is a ring homomorphism, and the image of $R$, denoted by $\phi(R)$, is a strongly $\phi$-ring.  The notion of \Prufer\ domains is one of the most famous integral domains that attract many algebraists. In 2004, Anderson and Badawi \cite{FA04} extended the notion of \Prufer\ domains to that of  \emph{ $\phi$-Pr\"{u}fer rings} which are $\phi$-rings $R$ satisfying that each finitely generated nonnil ideal is $\phi$-invertible.  The authors in \cite{FA04} characterized $\phi$-\Prufer\ rings from the  ring-theoretic viewpoint. In 2018, Zhao \cite{Z18} characterized  $\phi$-\Prufer\ rings using the homological properties of $\phi$-flat modules. Recently, Zhang and Qi \cite{ZQ20} gave a module-theoretic  characterization of $\phi$-\Prufer\ rings in terms of $\phi$-flat modules and nonnil-FP-injective modules.

Recall that an integral domain $R$ is called a \Prufer\ $v$-multiplication domain (\PvMD\ for short) provided that every nonzero ideal of $R$ is $w$-invertible (see \cite{LM99} for example). In 2014, Wang et al. \cite{KW14} showed that an integral domain $R$ is a $\PvMD$ if and only if $R_{\m}$ is a valuation domain  for any maximal $w$-ideal $\m$ of $R$. In 2015, Wang  et al. \cite{fq15} obtained that an integral domain $R$ is a $\PvMD$  if and only if $w$-w.gl.dim$(R) \leq 1$, if and only if every torsion-free $R$-module is $w$-flat. In 2018, Xing  et al. \cite{xw182} gave a new module-theoretic characterization of  \Prufer\ $v$-multiplication domains, i.e., an integral domain $R$ is a  \Prufer\ $v$-multiplication domain if and only if every divisible $R$-module is absolutely $w$-pure, if and only if every $h$-divisible $R$-module is absolutely $w$-pure. In order to extend the notion of \PvMD s to that of commutative rings in $\mathcal{H}$, the author of this paper and Zhao \cite{ZxlZ20} introduced the notion of $\phi$-$\PvMR$s as the $\phi$-rings in which any finitely generated nonnil ideal is $\phi$-$w$-invertible. They also gave some ring-theoretic and homology-theoretic characterizations of $\phi$-$\PvMR$s. In this note, we mainly study the module-theoretic characterizations of $\phi$-$\PvMR$s which can be seen a generalization of  Wang's and Xing's results in \cite{fq15} and \cite{xw182} respectively.

As our work involves the $w$-operation theory, we give a quick review as below. Let $R$ be a commutative ring and $J$ a finitely generated ideal of $R$. Then $J$ is called a \emph{$\GV$-ideal} if the natural homomorphism $R\rightarrow \Hom_R(J,R)$ is an isomorphism. The set of all $\GV$-ideals is denoted by $\GV(R)$.
Let $M$ be an $R$-module and define
\begin{center}
{\rm $\tor(M):=\{x\in M|Jx=0$, for some $J\in \GV(R) \}.$}
\end{center}
An $R$-module $M$ is said to be \emph{$\GV$-torsion} (resp., \emph{$\GV$-torsion-free}) if $\tor(M)=M$ (resp., $\tor(M)=0$). A $\GV$-torsion free module $M$ is said to be a \emph{$w$-module} if,  for any $x\in E(M)$, there is a $\GV$-ideal $J$ such that $Jx\subseteq M$ where $E(M)$ is the injective envelope of $M$. The \emph{$w$-envelope} $M_w$ of a $\GV$-torsion free module $M$ is defined  by the minimal $w$-module  that contains $M$.  A \emph{maximal $w$-ideal} for which is maximal among the $w$-submodules of $R$ is proved to be prime (see {\cite[Theorem 6.2.15]{fk16}}). The set of all maximal $w$-ideals is denoted by $w$-$\Max(R)$. Let $M$ be an $R$-module and set $\Lt(M) = (M/\tor(M))_w$. Recall from \cite{WK15} that $M$ is said to be $w$-projective if $\Ext^1_R(\Lt(M), N)$ is $\GV$-torsion for any torsion-free $w$-module $N$.

An $R$-homomorphism $f:M\rightarrow N$ is said to be a \emph{$w$-monomorphism} (resp., \emph{$w$-epimorphism}, \emph{$w$-isomorphism}) if for any $ \p\in w$-$\Max(R)$, $f_\p:M_\p\rightarrow N_\p$ is a monomorphism (resp., an epimorphism, an isomorphism). Note that $f$ is a $w$-monomorphism (resp., $w$-epimorphism) if and only if $\Ker(f)$ (resp., $\Coker(f)$) is $\GV$-torsion. A sequence $A\rightarrow B\rightarrow C$ is said to be \emph{$w$-exact} if,  for any $\p\in w$-$\Max(R)$, $A_\p\rightarrow B_\p\rightarrow C_\p$ is exact. A class $\mathcal{C}$ of $R$-modules is said to be \emph{closed under $w$-isomorphisms}  provided that for any $w$-isomorphism $f:M\rightarrow N$, if one of the modules $M$ and $N$ is in $\mathcal{C}$, so is the other. An $R$-module $M$ is said to be of \emph{finite type} if there exist a finitely generated free module $F$ and a $w$-epimorphism $g: F\rightarrow M$, or equivalently, if there exists a finitely generated $R$-submodule $N$ of $M$ such that $N_w = M_w$. Certainly, the class of  finite type modules is closed under $w$-isomorphisms.

\section{nonnil-absolutely $w$-pure modules}

Recall from \cite{xw181}, a $w$-exact sequence of $R$-modules $0\rightarrow N\rightarrow M\rightarrow L\rightarrow 0$ is said to be $w$-pure exact if, for any $R$-module $K$, the induced sequence $0\rightarrow K\otimes_RN\rightarrow K\otimes_RM\rightarrow K\otimes_RL\rightarrow 0$ is $w$-exact. If $N$ is a submodule of $M$ and the exact sequence  $0\rightarrow N\rightarrow M\rightarrow M/N\rightarrow 0$ is $w$-pure exact, then $N$ is said to be a $w$-pure submodule of $M$. Recall from \cite{xw182}, an $R$-module $M$ is called a \emph{absolutely $w$-pure module} provided that $M$ is   $w$-pure  in every module containing $M$ as a submodule.

Let $R$ be an $\NP$-ring and $M$  an $R$-module. Define
\begin{center}
$\phi$-$tor(M)=\{x\in M|Ix=0$ for some nonnil ideal $I$ of  $R \}$.
\end{center}
An $R$-module $M$ is said to be \emph{$\phi$-torsion} (resp., \emph{$\phi$-torsion free}) provided that  $\phi$-$tor(M)=M$ (resp., $\phi$-$tor(M)=0$).
Now we generalize the notions in \cite{xw181} and \cite{xw182} to $\NP$-rings. A $w$-exact sequence $0\rightarrow M\rightarrow N\rightarrow N/M\rightarrow 0$ of  $R$-modules is said to be  nonnil $w$-pure exact provided that $0\rightarrow \Hom_R(T,M)\rightarrow \Hom_R(T,N)\rightarrow \Hom_R(T,N/M)\rightarrow 0$ is $w$-exact for any finitely presented $\phi$-torsion module $T$. In addition, if $M$ is a submodule of $N$, then we say $M$ is   nonnil $w$-pure submodule in $N$.

\begin{definition} Let $R$ be an $\NP$-ring.  An $R$-module $M$ is called a \emph{nonnil-absolutely $w$-pure module} provided that $M$ is  nonnil  $w$-pure submodule  in every $R$-module containing $M$ as a submodule.
\end{definition}

Following from Xing \cite[Theorem 2.6]{xw182}, an $R$-module $M$ is absolutely $w$-pure if and only if $\Ext_R^1(F,M)$ is $\GV$-torsion for any finitely presented module $F$, if and only if $M$ is a  $w$-pure submodule in its injective envelope. Now, we give a $\phi$-version of Xing' result.

\begin{proposition}\label{int-ns}
Let $R$ be an $\NP$-ring and $M$ an $R$-module. The following statements are equivalent:
\begin{enumerate}
   \item  $M$ is a nonnil-absolutely  $w$-pure module;
      \item    $\Ext_R^1(T,M)$ is $\GV$-torsion for any finitely presented $\phi$-torsion module $T$;
    \item  $M$ is a nonnil  $w$-pure submodule in every injective module containing $M$;
    \item  $M$ is a nonnil  $w$-pure submodule in its injective envelope;
        \item  for any diagram $$\xymatrix@R=20pt@C=25pt{ &  M&\\
0\ar[r]&  K\ar[u]^{f}\ar[r]^{i}&F\ar@{-->}[lu]_{g_c}\\}$$
 with $F$ finitely generated projective, $K$ finitely generated and $F/K$ $\phi$-torsion, there is some $J\in\GV(R)$ such that any given $c\in J$, there exists $g_c:F\rightarrow M$ such that $cf=g_ci$.
\end{enumerate}
\end{proposition}\
\begin{proof} $(1)\Rightarrow (3)\Rightarrow (4):$ Trivial.

$(2)\Rightarrow (1):$ Let $N$ be an $R$-module  containing $M$, and $T$ a finitely presented $\phi$-torsion module. Then we have the following exact sequence $$0\rightarrow \Hom_R(T,M)\rightarrow \Hom_R(T,N)\rightarrow \Hom_R(T,N/M)\rightarrow \Ext^1_R(T,M).$$ Since $\Ext^1_R(T,M)$ is $\GV$-torsion, we have $$0\rightarrow \Hom_R(T,M)\rightarrow \Hom_R(T,N)\rightarrow \Hom_R(T,N/M)\rightarrow 0$$ is $w$-exact. Hence $M$ is a nonnil  $w$-pure submodule in $N$.

$(4)\Rightarrow (2):$ Let $E$ be the injective envelope of $M$. Then for any  finitely presented $\phi$-torsion module $T$, we have the following exact sequence:
$0\rightarrow \Hom_R(T,M)\rightarrow \Hom_R(T,E)\rightarrow \Hom_R(T,E/M)\rightarrow \Ext^1_R(T,M)\rightarrow 0$.  Thus  we have $\Ext^1_R(T,M)$ is $\GV$-torsion by $(4)$.

$(2)\Rightarrow (5):$ Consider the exact sequence $0\rightarrow K\xrightarrow{i} F\xrightarrow{\pi} F/K\rightarrow 0$ with  $F/K$  finitely presented  $\phi$-torsion. we have the following exact sequence: $\Hom_R(F,M)\xrightarrow{i^{\ast}} \Hom_R(K,M)\rightarrow \Ext_R^1(F/K,M)\rightarrow 0$. Since $F/K$ is finitely presented  $\phi$-torsion, $\Ext_R^1(F/K,M)$ is $\GV$-torsion by $(2)$. Thus $i^{\ast}$ is a $w$-epimorphism. Since $f\in \Hom_R(K,M)$, there exists a $\GV$-ideal $J$ of $R$ such that $Jf\in \Im (i^{\ast})$. So, for any given $c\in J$, there exists $g_c:F\rightarrow A$ such that $g_ci=cf.$

$(5)\Rightarrow (2):$ Let $T$ be a finitely presented $\phi$-torsion module. Then exists a short sequence $0\rightarrow K\xrightarrow{i} F\rightarrow T\rightarrow 0$ with $F$ finitely generated projective and $K$ finitely generated. Consider the exact sequence   $\Hom_R(F,M)\xrightarrow{i^{\ast}} \Hom_R(K,M)\rightarrow \Ext_R^1(T,M)\rightarrow 0$. For any $f\in \Hom_R(K,M)$, there is some $J\in\GV(R)$ such that any given $c\in J$, there exists $g_c:F\rightarrow M$ such that $cf=g_ci$ by $(5)$. So $Jf\subseteq \Im(i^{\ast})$. Thus $i^{\ast}$ is a $w$-epimorphism and so $\Ext_R^1(T,M)$ is  $\GV$-torsion.
\end{proof}

Recall from \cite[Definition 1.2]{ZQ20} that an $R$-module $M$ is called  \emph{nonnil-FP-injective} provided that $\Ext_R^1(T,M)=0$ for  any finitely presented $\phi$-torsion module $T$. Thus we have the following result by Proposition \ref{int-ns}.

\begin{corollary}\label{asfap-int-3}
Let $R$ be an $\NP$-ring.  Then every nonnil-FP-injective module is nonnil-absolutely $w$-pure.
\end{corollary}

\begin{lemma}\label{gv-abwp}
 Let $T$ be a $\GV$-torsion module. Then $T$ is a absolutely $w$-pure module.
\end{lemma}
\begin{proof} Let $T$ be a $\GV$-torsion module and $F$ a finitely presented $R$-module. Considering the exact sequence $0\rightarrow K\rightarrow P\rightarrow F\rightarrow 0$ with $P$ finitely generated projective and $K$ finitely generated,  we have the following exact sequence $\Hom_R(K,T)\rightarrow \Ext_R^1(F,T)\rightarrow 0$. Since $K$ is finitely generated and $T$ is $\GV$-torsion, $\Hom_R(K,T)$ is $\GV$-torsion. So $\Ext_R^1(F,T)$  is $\GV$-torsion. Consequently,  $T$  is a absolutely $w$-pure module.
\end{proof}

Obviously,  we have the following result by Proposition \ref{int-ns}, \cite[Theorem 2.6]{xw182} and Lemma \ref{gv-abwp}.
\begin{corollary}\label{asfap-int-1}
Let $R$ be an $\NP$-ring.  Then every absolutely $w$-pure module is nonnil-absolutely $w$-pure.  Consequently, every  $\GV$-torsion module is a nonnil-absolutely $w$-pure module.
\end{corollary}

In order to characterize rings over which every  nonnil-absolutely $w$-pure module is absolutely $w$-pure, we recall some basic facts.

\begin{lemma}\cite[Lemma 1.6]{ZxlZ20} \label{nin}
Let $R$ be a $\phi$-ring and  $I$ a nonnil ideal of $R$. Then $\Nil(R)=I\Nil(R)$.
\end{lemma}

\begin{lemma}\cite[Proposition 1.5]{ZQ20}\label{R-nil-1}
Let $R$ be a $\phi$-ring and $M$ an FP-injective $R/\Nil(R)$-module. Then $M$ is nonnil-FP-injective over $R$.
\end{lemma}

Let $R\{x\}$ be the $w$-Nagata ring of $R$, that is, the localization of $R[X]$ at the multiplicative closed set $S_w=\{f\in R[x]\,|\, c(f)\in \GV(R)\},$ where $c(f)$ is the content of $f$ (see \cite{WK15}).  Then $\{\m\{x\}|\m\in w$-$\Max(R)\}$ is the set of all maximal ideal of $R\{x\}$ by \cite[Proposition 3.3(4)]{WK15}. Set $$E'=\prod\limits_{m\in w\mbox{-}\Max(R)}E_R(R\{x\}/\m\{x\})$$ where $E_R(R\{x\}/\m\{x\})$ is the injective envelope of the $R$-module $R\{x\}/\m\{x\}$. Since $R\{x\}/\m\{x\}$ is a $w$-module over $R$ by \cite[Theorem 6.6.19(2)]{fk16}, then $E'$ is an injective $w$-module over $R$.  Set  $$\tilde{E}:=\Hom_R(R\{x\},E').$$ Then $\tilde{E}$ is trivially an $R\{x\}$-module. Since $R\{x\}$ is a flat $R$-module,   $\tilde{E}$ is an injective $w$-module by \cite[Theorem 6.1.18]{fk16} and \cite[Theorem 3.2.9]{EJ00}.
\begin{lemma}\cite[Corollary 3.11]{ZWQ202}\label{cog-tor}
Let $M$ be an $R$-module. The following statements are equivalent:
\begin{enumerate}
 \item  $M$ is $\GV$-torsion;
 \item $\Hom_R(M,E)=0$ for any injective $w$-module $E$;
 \item  $\Hom_R(M,\tilde{E})=0$.
\end{enumerate}
\end{lemma}

\begin{theorem}\label{asfap-int}
Let $R$ be  a $\phi$-ring.  Then  $R$ is an integral domain if and only if  any  nonnil-absolutely $w$-pure module is absolutely $w$-pure.
\end{theorem}
\begin{proof} If  $R$ is an integral domain, then any  nonnil-absolutely $w$-pure module is absolutely $w$-pure obviously.

On the other hand, we have $\Hom_{R}(R/\Nil(R),\widetilde{E})$ is an injective $R/\Nil(R)$-module  by \cite[Theorem 3.1.6]{EJ00}. Thus by Lemma \ref{R-nil-1}, $\Hom_{R}(R/\Nil(R),\widetilde{E})$  is a nonnil-FP-injective $R$-module, and so is a  nonnil-absolutely $w$-pure $R$-module. Thus we have  $\Hom_{R}(R/\Nil(R),\widetilde{E})$  is an absolutely $w$-pure $R$-module by assumption. That is,  $$\Ext_R^1(F, \Hom_{R}(R/\Nil(R),\widetilde{E}))\cong \Hom_{R}(\Tor_1^R(F, R/\Nil(R)),\widetilde{E})$$ is a $\GV$-torsion module for any finitely presented $R$-module $F$ as $\widetilde{E}$ is an injective $R$-module. Since $\widetilde{E}$ is a $w$-module, $\Hom_{R}(\Tor_1^R(F, R/\Nil(R)),\widetilde{E})$ is also a $w$-module by \cite[Theorem 6.1.18]{fk16}. Thus we have  $\Hom_{R}(\Tor_1^R(F, R/\Nil(R)),\widetilde{E})=0$. Hence $\Tor_1^R(F, R/\Nil(R))$ is $\GV$-torsion by Lemma \ref{cog-tor}. Let $s$ be a nilpotent element in $R$ and set $F=R/\langle s\rangle$. Then  $\Tor_1^R(F, R/\Nil(R))=\Tor_1^R(R/\langle s\rangle, R/\Nil(R))\cong\langle s\rangle\cap \Nil(R)/s\Nil(R)=\langle s\rangle/s\Nil(R)$ is $\GV$-torsion (see \cite[Exercise 3.20]{fk16}). Thus there is a $\GV$-ideal $J$ such that $sJ\subseteq s\Nil(R)$. Since $J$ is a  $\GV$-ideal, then it is a  nonnil ideal, thus $\Nil(R)=J\Nil(R)$ by Lemma \ref{nin}. So $sJ\subseteq s\Nil(R)=sJ\Nil(R)\subseteq sJ$. That is, $sJ=sJ\Nil(R)$.  Since $sJ$ is finitely generated, $sJ=0$ by Nakayama's lemma. Since $J\in \GV(R)$, $sR\subseteq R$ is $\GV$-torsion free, then $s=0$. Consequently, $\Nil(R)=0$ and so $R$ is an integral domain.
\end{proof}

\begin{lemma}\label{str-phi-lc}
Let $R$ be a ring. If $R$ is a $($strongly$)$ $\phi$-ring, then $R_{\p}$ is a $($strongly$)$ $\phi$-ring for any prime ideal $\p$ of $R$.
\end{lemma}
\begin{proof} Let $R$ be a $\phi$-ring and  $\p$ a  prime ideal  of $R$. Then $R_{\p}/\Nil(R_{\p})\cong (R/\Nil(R))_{\overline{\p}}$ which is certainly an integral domain. So $\Nil(R_{\p})$ is a prime ideal of  $R_{\p}$.
Let $\frac{r}{s}\in R_{\p}-\Nil(R_{\p})$ and $\frac{r_1}{s_1}\in \Nil(R_{\p})$. Note $r\in R-\Nil(R)$ and $r_1\in \Nil(R)$. Then $r_1=rt$ for some $t\in \Nil(R)$. Thus $\frac{r_1}{s_1}=\frac{rt}{s_1}=\frac{rts}{ss_1}=\frac{r}{s}\frac{ts}{s_1}\in \langle \frac{r}{s}\rangle$. So  $\Nil(R_{\p})$ is a divided prime ideal of  $R_{\p}$. Hence $R_{\p}$ is a $\phi$-ring.
Now suppose $R$ is a strongly $\phi$-ring. Let $\frac{r}{s}\in R_{\p}-\Nil(R_{\p})$. Then $r$ is non-nilpotent, and thus $r$ is regular. Assume  $\frac{r}{s}\frac{r_1}{s_1}=0$ in $R_{\p}$. Then there exists $t\in R-\p$ such that $rr_1t=0$. Thus $r_1t=0$. Hence $r_1$ and thus $\frac{r_1}{s_1}$ is equal to $0$ since $t$ is also regular. Consequently, $R_{\p}$ is a strongly $\phi$-ring.
\end{proof}
\begin{remark} Note that the converse of Lemma \ref{str-phi-lc} is not true in general. Indeed, let $R$ be a von Neumann regular ring which is not a field. Then $R_{\p}$ is a field for any prime ideal $\p$ of $R$. However, $R$ is not a $\phi$-ring since $\Nil(R)=0$ is not a prime ideal in this case.
\end{remark}

Let $R$ be an $\NP$-ring. Recall from \cite{ZWT13} that an $R$-module $M$ is said to be \emph{$\phi$-flat} if for every monomorphism $f : A\rightarrow B$ with $\Coker(f)$ $\phi$-torsion,  $f\otimes_R 1 : A\otimes_R M \rightarrow B\otimes_R M$ is a monomorphism; a $\phi$-ring $R$ is said to be $\phi$-von Neumann if every $R$-module is $\phi$-flat. The authors in  \cite[Theorem 4.1]{ZWT13} proved that a $\phi$-ring $R$  is $\phi$-von Neumann if and only if  the Krull dimension of $R$ is $0$. It was also shown in \cite[Theorem 1.7]{ZQ20} that a $\phi$-ring $R$ is $\phi$-von Neumann if and only if $R/\Nil(R)$ is a field, if and only if every non-nilpotent element is invertible, if and only if every $R$-module is nonnil-FP-injective. Recall from \cite[Definition 1.3]{ZxlZ20} that an $R$-module $M$ is said to be \emph{$\phi$-$w$-flat} if, for every monomorphism $f : A\rightarrow B$ with $\Coker(f)$ $\phi$-torsion,  $f\otimes_R 1 : A\otimes_R M \rightarrow B\otimes_R M$ is a $w$-monomorphism. It was proved in \cite[Theorem 3.1]{ZxlZ20} that  a $\phi$-ring $R$ is $\phi$-von Neumann if and only if every $R$-module is $\phi$-$w$-flat. Now we give a new characterization of $\phi$-von Neumann rings.
\begin{lemma}\label{asfapfvnrl}
Let $R$ be  a $\phi$-ring.  Then  $R$ is a $\phi$-von Neumann regular ring  if and only if  $R_{\m}$ is a $\phi$-von Neumann regular ring  for any $\m\in w$-$Max(R)$.
\end{lemma}
\begin{proof} Let $R$ be a $\phi$-von Neumann regular ring and  $\m$  a prime ideal. Let $\frac{r}{s}$ be a non-nilpotent element in $R_\m$. Then $r$ is non-nilpotent. So $r$ is invertible by \cite[Theorem 1.7]{ZQ20}. Hence $\frac{r}{s}$ is also invertible in  $R_\m$ implying  $R_{\m}$ is a $\phi$-von Neumann regular ring by \cite[Theorem 1.7]{ZQ20} and Lemma \ref{str-phi-lc}.

Now let $r$ be non-nilpotent element in $R$. Then $\frac{r}{1}$ is a non-nilpotent element in $R_{\m}$ for any $\m\in w$-$Max(R)$, since $R$ is a $\phi$-ring. By \cite[Theorem 1.7]{ZQ20}, $\frac{r}{1}$  is invertible in $R_{\m}$. Thus $r\not\in \m$ for any $\m\in w$-$\Max(R)$. So $\langle r\rangle_w=R$, and hence $r$ is   invertible  by \cite[Exercise 6.11(2)]{fk16}.
\end{proof}

\begin{theorem}\label{asfap-int-2}
Let $R$ be  a $\phi$-ring.  Then  $R$ is a $\phi$-von Neumann regular ring  if and only if  every  $R$-module is nonnil-absolutely $w$-pure.
\end{theorem}

\begin{proof} Suppose $R$ is a $\phi$-von Neumann regular ring and $M$ is an $R$-module. Then $R/\Nil(R)$ is a field. By \cite[Theorem 3.3]{ZxlZ20}  $R$ is a $\phi$-\Prufer\ $v$-multiplication ring  and thus $R_{\m}$ is a $\phi$-chained ring for any maximal $w$-ideal $\m$ of $R$.
Let $T$ be a finitely presented $\phi$-torsion module. Then $\Ext_R^1(T,M)_{\m}=\Ext_{R_{\m}}^1(T_{\m},M_{\m})=\oplus_{i=1}^n\Ext_{R_{\m}}^1( R_{\m}/R_{\m}x_i, M_{\m})$ for some non-nilpotent element $x_i\in R_{\m}$ by  \cite[Theorem 3.9.11]{fk16}. By Lemma \ref{asfapfvnrl}, $R_{\m}$ is a $\phi$-von Neumann regular ring. Thus $x_i$ is an invertible element by \cite[Theorem 1.7]{ZQ20}. So $R_{\m}/R_{\m}x_i=0$ and thus $\Ext_{R_{\m}}^1( R_{\m}/R_{\m}x_i, M_{\m})=0$. Hence $\Ext_R^1(T,M)_{\m}=0$ for any $\m\in w$-$\Max(R)$. It follows that $\Ext_R^1(T,M)$ is $\GV$-torsion. Consequently, $M$ is nonnil-absolutely $w$-pure.

On the other hand, let $I$ be a finitely generated nonnil ideal of $R$. Since for any $R$-module $M$, $\Ext_R^1(R/I,M)$  is  $\GV$-torsion, then $R/I$ is finitely generated $w$-projective. $R_{\m}/I_{\m}$ is a finitely generated projective $R_{\m}$-module for any $\m\in w$-$\Max(R)$ by \cite[Theorem 6.7.18]{fk16}. Then $I_{\m}$ is an idempotent ideal of $R_{\m}$ by \cite[Theorem 1.2.15]{g}. By \cite[Chapter I, Proposition 1.10]{FS01}, $I_{\m}$ is generated by an idempotent $e_{\m}\in R_{\m}$. Thus $R_{\m}$ is a $\phi$-von Neumann regular ring by \cite[Theorem 4.1]{ZWT13} and Lemma \ref{str-phi-lc}.
 So $R$ is $\phi$-von Neumann regular by Lemma \ref{asfapfvnrl}.
\end{proof}

\section{Some new characterizations of $\phi$-\Prufer\ $v$-multiplication rings}

Recall from \cite{A01} that a $\phi$-ring  $R$ is said to be a \emph{$\phi$-chain ring} ($\phi$-CR for short) if for any $a,b\in R-\Nil(R)$, either $a|b$ or $b|a$ in $R$. A $\phi$-ring  $R$ is said to be a \emph{$\phi$-\Prufer\ ring} if every finitely generated nonnil ideal $I$ is $\phi$-invertible, i.e., $\phi(I)\phi(I^{-1})=\phi(R)$ where $I^{-1}=\{x\in\T(R)|Ix\subseteq R\}$. It follows from \cite[Corollary 2.10]{FA04} that a $\phi$-ring $R$ is $\phi$-\Prufer, if and only if $R_{\fkm}$ is a $\phi$-CR for any maximal ideal $\fkm$ of $R$, if and only if $R/\Nil(R)$ is a \Prufer\ domain, if and only if $\phi(R)$ is a \Prufer\ ring.

Let $R$ be a $\phi$-ring.  Recall from \cite{kf12} that a nonnil ideal $J$ of $R$ is said to be a \emph{$\phi$-$\GV$-ideal} (resp., \emph{$\phi$-$w$-ideal}) of $R$ if $\phi(J)$ is a $\GV$-ideal (resp., $w$-ideal) of $\phi(R)$.  An ideal $I$ of $R$ is  \emph{$\phi$-$w$-invertible} if $(\phi(I)\phi(I)^{-1})_W=\phi(R)$ where $W$ is the $w$-operation of $\phi(R)$.  In order to extend $\PvMD$s to $\phi$-rings, the authors in \cite{ZxlZ20} gave the notion of $\phi$-\Prufer\ $v$-multiplication rings: a $\phi$-ring $R$ is said to be a \emph{$\phi$-\Prufer\ $v$-multiplication ring} ($\phi$-$\PvMR$ for short) provided that any finitely generated nonnil ideal is $\phi$-$w$-invertible. They also show that a $\phi$-ring $R$ is a $\phi$-$\PvMR$ if and only if  $R_{\fkm}$ is a $\phi$-CR for any $\fkm\in w$-$Max(R)$, if and only if  $R/\Nil(R)$ is a $\PvMD$, if and only if $\phi(R)$ is a $\PvMR$.

Recall that an $R$-module $E$ is said to be  \emph{divisible} if $sM=M$ for any regular element $s\in R$, and an $R$-module $M$ is said to be  \emph{$h$-divisible} provided that $M$ is a quotient of an injective module. Evidently, any injective module is $h$-divisible and any $h$-divisible module is divisible. The author in \cite{ZQ20} introduced the notion of  \emph{nonnil-divisible} modules $E$ in which for any $m\in E$ and any non-nilpotent element $a\in R$, there exists  $x\in E$ such that $ax=m$.

\begin{lemma}\cite[Lemma 2.2]{ZQ20}\label{nonnil-div-ext}
Let $R$ be an $\NP$-ring and $E$ an $R$-module. Consider the following statements:
 \begin{enumerate}
   \item $E$ is nonnil-divisible;
   \item $E$ is divisible;
  \item $\Ext_R^1(R/\langle a\rangle,E)=0$ for any $a\not\in\Nil(R)$.
 \end{enumerate}
Then we have $(1)\Rightarrow (2)$ and $(1)\Rightarrow (3)$. Moreover, if $R$ is a $\ZN$-ring, all  statements are equivalent.
\end{lemma}

\begin{lemma}\cite[Lemma 2.4]{ZQ20}\label{local nonnil-div}
Let $R$ be an $\NP$-ring and $E$ a nonnil-divisible  $R$-module. Then $E_{\p}$ is a nonnil-divisible  $R_{\p}$-module for any  prime ideal $\p$ of $R$.
\end{lemma}

Let $M$ be an $R$-module. Recall from \cite{WK15} that  $M$ is said to have $w$-rank $n$ if, for any maximal $w$-ideal $\m$ of $R$, $M_\m$ is a free $R_\m$-module of
rank $n$. Let $\tau$ denote the trace map of $M$, that is, $\tau: M\otimes_R \Hom(M,R)\rightarrow R$ defined by
$\tau(x\otimes f)=f(x)$ for $x\in M$ and $f\in M$. $M$ is said to be $w$-invertible, if the trace map $\tau$ is a $w$-isomorphism. It was proved in  \cite[Theorem 4.13]{WK15} that  an $R$-module $M$ is $w$-invertible if and only if $M$ is of finite type and has $w$-rank $1$, if and only if $M$ is $w$-projective of finite type and has $w$-rank 1.

\begin{proposition}\label{pro-inv}
Let $R$ be a strongly $\phi$-ring and $I$ a finitely generated nonnil ideal of $R$. If $I$ is $w$-projective, then $I$ is $\phi$-$w$-invertible.
\end{proposition}
\begin{proof} Let  $I$ a finitely generated nonnil ideal of the strongly $\phi$-ring $R$. Then $I$ is a regular ideal of $R$. Let $\m$ be a maximal $w$-ideal of $R$. Since $I$ is $w$-projective  $R$-ideal, $I_\m$ is a free ideal of $R_\m$ by \cite[Theorem 6.7.11]{fk16}.  Then $I_\m\cong R_\m$ or $I_\m=0$. We claim that $I_\m\cong R_\m$. Indeed, let $r$ be a regular element in $I$. If $I_\m=0$, then there is an element $s\in R-\m$ such that $rs=0$. So $s=0$ which is a contradiction. Hence, $I_\m$ is of rank $1$ for any maximal $w$-ideal $\m$ of $R$. By \cite[Theorem 4.13]{WK15},  $\phi(I)=I$ is $w$-invertible since $R$ is a strongly $\phi$-ring. Hence, $I$ is $\phi$-$w$-invertible.
\end{proof}

\begin{lemma}\cite[Proposition 2.12]{ZQ20}\label{w-phi-tor}
Let $R$ be an $\NP$-ring, $\p$  a  prime ideal of $R$ and $M$ an $R$-module. Then $M$ is $\phi$-torsion over $R$ if and only $M_{\p}$ is $\phi$-torsion over $R_{\p}$.
\end{lemma}

\begin{lemma}\label{w-phi-tor-free}
Let $R$ be an $\NP$-ring,  $M$ an $R$-module. Suppose $M$ is $\phi$-torsion free over $R$, $M_{\m}$ is $\phi$-torsion free over $R_{\m}$ for any  maximal $w$-ideal $\m$ of $R$. Moreover, if $M$ is $\GV$-torsion free, then the converse also holds.
\end{lemma}
\begin{proof} Suppose $M$ is a $\phi$-torsion free $R$-module. Let $\m$ be a maximal $w$-ideal of $R$ and $\frac{m}{s}\in M_{\m}$. Suppose $I_{\m}$ is a nonnil ideal of $R_{\m}$ and $I_{\m}\frac{m}{s}=0$ in $M_{\m}$. then there exists $t\not\in \m$ such that $tIm=0$ in $R$. Since $I$ is nonnil in $R$ by  \cite[Lemma 1.1]{ZxlZ20}, we have $It$ is also nonnil as $t$ is non-nilpotent. Since $M$ be an $\phi$-torsion free, $m$ and thus $\frac{m}{s}$ is equal to $0$.

Suppose $M$ is a $\GV$-torsion free  $R$-module such that $M_{\m}$ is $\phi$-torsion free over $R_{\m}$ for any  maximal $w$-ideal $\m$ of $R$. Let $m\in M$ such that $Im=0$ for some nonnil ideal $I$ of $R$. Then $I_{\m}\frac{m}{1}=0$ in $M_{\m}$. Since $I_{\m}$ is nonnil in $R_{\m}$ by  \cite[Lemma 1.1]{ZxlZ20}, $\langle m\rangle_{\m}=0$ for any maximal $w$-ideal $\m$ of $R$. Thus $\langle m\rangle$ is $\GV$-torsion in $M$ by \cite[Theorem 6.2.15]{fk16}.  Since $M$ is $\GV$-torsion free by assumption, we have  $m=0$.
\end{proof}

It is well-known that an integral domain $R$ is a $\PvMD$ if and only if every torsion-free $R$-module is $w$-flat, if and only if every ($h$-)divisible $R$-module is absolutely $w$-pure (see \cite{fq15,xw182}). Recently, the authors in \cite{ZQ20} characterized  $\phi$-\Prufer\ rings in terms of nonnil-FP-injective modules, that is,  a strongly $\phi$-ring $R$ is a $\phi$-\Prufer\ ring if and only if any $\phi$-torsion free $R$-module is $\phi$-flat, if and only if  any  ($h$-)divisible module is nonnil-FP-injective. Now, we characterize $\phi$-$\PvMR$s in terms of  $\phi$-$w$-flat modules, nonnil-absolutely $w$-pure modules and $w$-projective modules, which can be seen as a generalization of the results in \cite{fq15,xw182,ZQ20}.
\begin{theorem}\label{w-g-flat-1}
Let $R$ be a strongly $\phi$-ring. The following statements are equivalent for $R$:
\begin{enumerate}
   \item $R$ is a $\phi$-$\PvMR$;
    \item  any $\phi$-torsion free $R$-module is $\phi$-$w$-flat;
    \item  any nonnil ideal of $R$ is $w$-flat;
    \item  any ideal of $R$ is $\phi$-$w$-flat;

     \item   any  divisible module is  nonnil-absolutely $w$-pure module;

      \item   any $h$-divisible module is  nonnil-absolutely $w$-pure module;

      \item  any finitely generated nonnil ideal of $R$ is $w$-projective;

       \item  any finite type nonnil ideal of $R$ is $w$-projective.
\end{enumerate}
\end{theorem}\

\begin{proof}  $(1)\Rightarrow (2)$: Let $\m$ be a maximal $w$-ideal of $R$, $M$ a $\phi$-torsion free $R$-module.  By Lemma \ref{w-phi-tor-free}, $M_{\m}$ is $\phi$-torsion free over $R_{\m}$. Since $R$ is a $\phi$-$\PvMR$, $R_{\m}$ is a $\phi$-CR by \cite[Theorem 3.3]{ZxlZ20}. Then $M_{\m}$ is $\phi$-flat by \cite[Theorem 4.3]{Z18}, and thus $M$ is $\phi$-$w$-flat by \cite[Theorem 1.4]{ZxlZ20}.

$(2)\Rightarrow (4)$: It follows from $R$ is $\phi$-torsion free since $R$ is a strongly $\phi$-ring (see \cite[Proposition 2.2]{Z18}).

$(4)\Leftrightarrow (3)$:  Let $J$ be a nonnil ideal of $R$, $I$ an ideal of $R$. We have
$$\Tor_1^R(R/J,I)\cong \Tor_2^R(R/J,R/I)\cong \Tor_1^R(R/I,J).$$ The result follows the statement that  any ideal of $R$ is $\phi$-$w$-flat is equivalence to the statement that  any nonnil ideal of $R$  is $w$-flat.

$(4)\Rightarrow (1)$: See \cite[Theorem 3.8]{ZxlZ20}.

$(1)\Rightarrow(5)$: Let  $T$ be a finitely presented $\phi$-torsion module and $\m$  a maximal $w$-ideal of $R$. Then by Lemma \ref{w-phi-tor}, $T_{\m}$ is a finitely presented $\phi$-torsion $R_{\m}$-module.
By \cite[Theorem 3.3]{ZxlZ20}, $R_{\m}$ is a $\phi$-chained ring. Thus,  by \cite[Theorem 4.1]{Z18}, $T_{\m}\cong \oplus_{i=1}^n R_{\m}/R_{\m}x_i$ for some regular element $x_i\in R_{\m}$ as $R_{\m}$ is a strongly $\phi$-ring by Lemma \ref{str-phi-lc}. Let $E$ be a divisible module. Then $E_{\m}$ is a divisible module over $R_{\m}$ by Lemma \ref{nonnil-div-ext} and Lemma \ref{local nonnil-div}.
 Thus $\Ext_R^1(T,E)_{\m}=\Ext_{R_{\m}}^1(T_{\m},E_{\m})=\oplus_{i=1}^n\Ext_{R_{\m}}^1( R_{\m}/R_{\m}x_i, E_{\m})=0$ by  Lemma \ref{nonnil-div-ext} and \cite[Theorem 3.9.11]{fk16}. It follows that $\Ext_R^1(T,E)$ is a $\GV$-torsion module. Therefore, $E$ is a nonnil-absolutely $w$-pure module.

$(5)\Rightarrow  (6)$ and $(8)\Rightarrow(7)$: Trivial.

$(6)\Rightarrow(7)$: Let $N$ be an $R$-module, $I$ a finitely generated nonnil ideal of $R$. The short exact sequence $0\rightarrow I\rightarrow R \rightarrow R/I\rightarrow 0$ induces a long exact sequence as follows: $$0=\Ext_R^1(R,N)\rightarrow \Ext_R^1(I,N)\rightarrow \Ext_R^2(R/I,N)\rightarrow \Ext_R^2(R,N)=0.$$
Let $0\rightarrow N\rightarrow E\rightarrow K\rightarrow 0$ be an exact sequence where $E$ is the  injective envelope of $N$. There exists a long exact sequence as follows:
$$0=\Ext_R^1(R/I,E)\rightarrow \Ext_R^1(R/I,K)\rightarrow \Ext_R^2(R/I,N)\rightarrow \Ext_R^2(R/I,E)=0.$$
Thus $\Ext_R^1(I,N)\cong \Ext_R^2(R/I,N)\cong \Ext_R^1(R/I,K)$ is a $\GV$-torsion module as $K$ is nonnil-absolutely $w$-pure by $(6)$. It follows that $I$ is a $w$-projective ideal of $R$ by \cite[Corollary 2.5]{fq21}.

$(7)\Rightarrow(1)$: It follows from Proposition \ref{pro-inv}.

$(7)\Rightarrow(8)$: Let $I$ be a finite type nonnil ideal of $R$, then there is a finitely generated sub-ideal $K$ of $I$ such that $K/I$ is $GV$-torsion (see \cite[Proposition 6.4.2(3)]{fk16}). Then $I$ is $w$-isomorphic to $K$. We claim that $K$ is a nonnil ideal. Indeed,  since $I$ is nonnil, there is an non-nilpotent element $s\in I$. Thus there is a $\GV$-ideal $J$ of $R$ such that $Js\subseteq K$. Since $J$ is nonnil and $R$ is a $\phi$-ring, we have $K$ is a nonnil ideal of $R$. By (7), $K$ is $w$-projective. And thus $I$ is $w$-projective by \cite[Proposition 6.7.8(1)]{fk16}.
\end{proof}

Obviously, every  nonnil-FP-injective module is nonnil-absolutely $w$-pure. The following example shows that the converse does not hold in general.
\begin{example} Let $D$ be a $\PvMD$ but not a \Prufer\ domain, $K$ the quotient field of $D$. Then the idealization $R=D(+)K$ is  a  $\phi$-$\PvMR$ but not a $\phi$-\Prufer\ ring. Note that $R$ is a strongly $\phi$-ring by \cite[Remark 1]{FA05}. Thus there is a nonnil-absolutely $w$-pure divisible module $M$ which is not  nonnil-FP-injective by Theorem \ref{w-g-flat-1} and \cite[Theorem 2.13]{ZQ20}.
\end{example}

\begin{acknowledgement}\quad\\
The first author was supported by the Natural Science Foundation of Chengdu Aeronautic Polytechnic (No. 062026)  and the National Natural Science Foundation of China (No. 12061001).
\end{acknowledgement}


\bigskip
\begin{thebibliography}{99}
\bibitem{FA04} D. F. Anderson, A. Badawi, {\it On $\phi$-Pr\"{u}fer rings and $\phi$-Bezout rings},  Houston J. Math.  \textbf{30} (2004), 331-343.

\bibitem{FA05}D. F. Anderson, A. Badawi,  {\it  On $\phi$-Dedekind rings and $\phi$-Krull rings},  Houston J. Math.  \textbf{31}  (2005), 1007-1022.


\bibitem{A97} A. Badawi,  {\it On divided commutative rings},  Comm. Algebra  \textbf{27} (1999), 1465-1474.

\bibitem{A01} A. Badawi, {\it On $\phi$-chained rings and $\phi$-pseudo-valuation rings},  Houston J. Math.  \textbf{27} (2001), 725-736.

\bibitem{EJ00}  E. E. Enochs, O. M. G. Jenda,   {\it  Relative homological algebra}, De Gruyter Exp. Math., vol.  {\bf 30}. Berlin: Walter de Gruyter Co, 2011.

\bibitem{FS01} L. Fuchs,  L. Salce,  {\it Modules over Non-Noetherian Domains}, New York: Math Surveys and Monographs, 84, AMS, 2001.

\bibitem{g} S. Glaz,  {\it  Commutative Coherent Rings}, Lecture Notes in Mathematics, vol.  {\bf 1371}, Berlin: Spring-Verlag, 1989.


\bibitem{kf12}  H. Kim ,F. G. Wang,   {\it On $\phi$-strong Mori rings},  Houston J. Math.  {\bf 38} (2012), no. 2,  359-371.


\bibitem{KW14} F. G. Wang,  H. Kim,  {\it $w$-injective modules and $w$-semi-hereditary rings},  J. Korean Math. Soc.  {\bf 51} (2014),  no. 3, 509-525.

\bibitem{WK15} F. G. Wang, H. Kim,   {\it Two generalizations of projective modules and their applications},  J. Pure Appl. Algebra  {\bf 219} (2015), no. 6,  2099-2123.

\bibitem{fk16} F. G. Wang,  H. Kim,  {\it  Foundations of Commutative rings and Their Modules}, Singapore: Springer, 2016.

\bibitem{LM99}    F. G. Wang and R. L. McCasland,{\it On strong Mori domains}, J. Pure Appl.
Algebra, {\bf 135} (1999), no. 2,  155-165.

\bibitem{fq15} F. G. Wang, L. Qiao,    {\it  The $w$-weak global dimension of commutative rings}, Bull. Korean Math. Soc.  {\bf 52} (2015), no. 4, 1327-1338.

\bibitem{fq21}  F. G. Wang, L. Qiao, {\it A new version of a theorem of Kaplansky}, Comm. Algebra  {\bf 48} (2020), no. 8, 3415-3428.


\bibitem{xw181} S.Q. Xing and F.G. Wang, {\it Purity over Pr\"{u}fer $v$-multiplication domains}, J. Algebra Appl.  {\bf 16} (2018), no 5, 1850100 (11 pages).

\bibitem{xw182}  S.Q. Xing and F.G. Wang, {\it Purity over Pr\"{u}fer $v$-multiplication domains, II}. J. Algebra Appl. {\bf 16} (2018), no 6, 1850223 (11pages).


\bibitem{ZQ20} X. L. Zhang,  W. Qi, {\it Some Remarks on $\phi$-Dedekind rings and $\phi$-Pr\"{u}fer rings}, https://arxiv.org/abs/2103.08278.


\bibitem{ZWQ202} X. L. Zhang,  F. G. Wang,  {\it On characterizations of $w$-coherent rings II},   Comm. Algebra  {\bf 58} (2021), no 4, 1039-1052.


\bibitem{ZxlZ20} X. L. Zhang,  W. Zhao,  {\it On $w$-$\phi$-flat modules and their homological dimensions},   Bull. Korean Math. Soc.   {\bf 58} (2021), no. 4, 1039-1052.


\bibitem{Z18} W. Zhao,   {\it On $\phi$-flat modules and $\phi$-\Prufer\ rings},  J. Korean Math. Soc.   {\bf 55} (2018), no. 5, 1221-1233.

\bibitem{ZWT13} W. Zhao,  F. G. Wang and G. H. Tang,   {\it On $\phi$-von Neumann regular rings},  J. Korean Math. Soc., {\bf 50} (2013), no. 1, 219-229.
\end{thebibliography}
\end{document}